\documentclass[12pt, reqno]{amsart}
\usepackage{amsmath,amssymb,amsbsy,amsfonts, times, amsthm}
\usepackage[foot]{amsaddr}
\usepackage{lipsum}
\usepackage{color}
\usepackage{graphics}
\usepackage{epsfig,psfrag,graphicx}
\usepackage{verbatim}
\usepackage{cancel}
\usepackage[normalem]{ulem}
\usepackage{colonequals}

\usepackage{hyperref}

\usepackage{parskip}
\usepackage{eepic,epic}
\usepackage{pgf,tikz}
\usetikzlibrary{arrows}
\usetikzlibrary{shapes,decorations}

\usepackage{multirow}

\newcommand{\cred}{\color{black}}
\newcommand{\cblue}{\color{black}}

\newcommand{\pair}[2]{\left\langle #1 , #2 \right\rangle}

\newcommand{\RR}{\mathbb{R}}
\newcommand{\RRD}{{\mathbb{R}^d}}

\newcommand{\CC}{\mathcal{C}}
\newcommand{\CCA}{\mathcal{C}^{1+\alpha}}

\newtheorem{definition}{Definition}[section]
\newtheorem{theorem}{Theorem}[section]
\newtheorem{lemma}[theorem]{Lemma}
\newtheorem{prop}[theorem]{Proposition}

\theoremstyle{remark}
\newtheorem*{remark}{Remark}

\newcommand{\nc}{\newcommand}
\nc{\PP}{{\cal P}} 
\nc{\V}{{\cal V}} 
\nc{\M}{{\cal M}} 
\nc{\T}{{{\mathbb R}^n}}  
\nc{\D}{{\cal D}} 
\nc{\R}{{\mathbb R}} 
\nc{\N}{{\mathbb N}}
\nc{\Thh}{T_h(\mu_t^h)}
\nc{\Th}{T_h(\mu_t)}

\newcommand{\dd}{{\,\mathrm{d}}}

\newcommand{\NN}{\mathbb{N}}

\newcommand\footnoteref[1]{\protected@xdef\@thefnmark{\ref{#1}}\@footnotemark}

\nc{\weak}{\rightharpoonup}
\nc{\weakstar}{\stackrel{\ast}{\rightharpoonup}} 

\renewcommand{\div}{{\mathrm{div}_x}\,}

\def\vec#1{\boldsymbol{#1}}

\def\dd#1#2{\frac{\d #1}{\d #2}}

\def\b0{\vec{0}}

\renewcommand{\div}{{\rm div}}

\newcommand{\Eehj}{e^{\int_0^t \overline{w}(s, y; h_1)\dd s}}
\newcommand{\Eehd}{e^{\int_0^t \overline{w}(s, y; h_2)\dd s}}
\newcommand{\Eehz}{e^{\int_0^t \overline{w}(s, y; 0)\dd s}}
\newcommand{\phiz}{\varphi(X(t, y; 0))}

\def\XXint#1#2#3{{\setbox0=\hbox{$#1{#2#3}{\int}$ }
\vcenter{\hbox{$#2#3$ }}\kern-.6\wd0}}

\def\dd{{\rm d}}

\title[Measure solution to transport equation, \textnormal{K. \L yczek, k.lyczek@mimuw.edu.pl}]{Differentiability in perturbation parameter of measure solutions to perturbed transport equation}

\author{Piotr Gwiazda$^1$}
\author{Sander C. Hille$^2$}
\author{Kamila \L yczek$^3$}
\author{Agnieszka \'{S}wierczewska-Gwiazda$^3$}

\address{{\it 2010 Mathematics Subject Classification.} \textnormal{Primary: 35Q93; Secondary: 28A33.}}
\address{$^1$Institute of Mathematics, Polish Academy of Sciences, \'Sniadeckich 8, 00-656 Warszawa, Poland}
\email{pgwiazda@mimuw.edu.pl}

\address{$^2$Mathematical Institute, Leiden University
P.O. Box 9512, 2300 RA Leiden, The Netherlands}
\email{shille@math.leidenuniv.nl}

\address{$^3$Institute of Applied Mathematics and Mechanics, University of Warsaw, Banacha 2,\newline 02-097 Warszawa, Poland}
\email{k.lyczek@mimuw.edu.pl}
\email{aswiercz@mimuw.edu.pl}
\begin{document}
\maketitle

\begin{abstract}
We consider a linear perturbation in the velocity field of  the transport equation. We investigate solutions in the space of bounded Radon measures and show that they are differentiable with respect to  the perturbation parameter in a proper  Banach space, which is predual to  the H\"older space $\CCA(\RRD)$.  This result on differentiability  is necessary for application in~optimal control theory, which we also discuss. \end{abstract}

\section{Introduction}

Analysis of perturbations in partial differential equation systems is an important issue. {\it Structured population models} \cite{Carillo:2015, Czochra:2010,Czochra:2010_2,carillo:2012,Rosi:2016},
 dynamics of system \cite{Carillo:2011,Gango:2013,Herty:2014,Fornasier:2018a,Fornasier:2018} and {\it vehicular traffic flow} \cite{Degond:2008,Evers:2016,Goatin:2016,Goatin:2017} were investigated for Lipschitz dependence on initial conditions in space of measures. However, the differentiability (not only Lipschitz dependence) is necessary for the application in optimal control theory or linearised stability. Previous considerations concerning the transport equation in  the space of measures did not allow to analyse the differentiability of~solutions with respect to a perturbation of the system \cite{Ambrosio:2008,Thieme:2003,Piccoli2014}.

In this paper we consider solutions to {\cred a} perturbed transport equation {\cred in the space of bounded Radon measures, denoted by $\mathcal{M}(\RRD)$},  {\cred where the}  perturbation {\cred is linear in the} velocity field.

Consider the initial value problem for {\cred the} transport equation in conservative form
\begin{equation}\label{def:prob_init}
\left\{ \begin{array}{l}
\partial_t\mu_t + \textrm{div}_x(b\mu_t) =  w \mu_t \quad \textrm{in }(\CC_c^1([0,\infty) \times \mathbb{R}^d))^{\ast}, \\
\mu_{t=0}= \mu_0 \in \mathcal{P}(\mathbb{R}^d), 
\end{array} \right. 
\end{equation}
where the velocity field $\left(t \mapsto b(t, \cdot) \right) \in \CC^0(\left[ 0, +\infty)\right.; \CCA(\RRD))$, the initial condition is a probability measure on $\RRD$ and $w(t,x)\in\CCA([0, \infty) \times \RRD)$. By $(\cdot)^{\ast}$ we denote the topological dual to $(\cdot)$, when the latter is equipped with a suitable locally convex or norm topology; $\mathcal{C}_c^1$ is the space of continuous functions with compact support and $\CCA$ is the space of functions of which first order partial derivatives are H\"older continuous with exponent $\alpha$, where $0 < \alpha \leq 1$.

Existence and uniqueness of solutions to equation (\ref{def:prob_init}) was proved in \cite{MANIGLIA2007601}, see Lemma~\ref{lem:repr_form}. The solution $\mu_t\colon  [0, \infty) \to \mathcal{P}(\mathbb{R}^d)$ is a {\it narrowly continuous} curve (by \cite{MANIGLIA2007601}, Lemma 3.2). Recall  that a mapping $ [0, \infty)\ni t \mapsto \mu_t \in \mathcal{P}(\mathbb{R}^d)$ is  narrowly continuous if 
$t \mapsto \int_{\mathbb{R}^d} \eta \dd \mu_t$ is a~continuous function for all $\eta$ in the space of continuous and bounded functions defined on $\RRD$, $\CC_b(\mathbb{R}^d)$.

We start by defining a weak solution to equation \eqref{def:prob_init}.

\begin{definition}\label{def:weak_solu}
Let $\mu_0 \in \mathcal{P}(\mathbb{R}^d)$ 
 and $\left(t \mapsto b(t, \cdot) \right) \in \CC^0(\left[ 0, +\infty)\right.; \CCA(\RRD))$.\\
 We say that {\cred the} narrowly continuous curve $t\mapsto \mu_t\in \mathcal{P}(\RRD)$ is a weak solution to (\ref{def:prob_init})  if 
\begin{equation}\label{eqn:weak_solu}
\begin{split}
\int_0^{\infty} \int_{\RRD} \left(\partial_t\varphi(t,x) +b\nabla_x \varphi(t,x)\right) \dd \mu_t(x) \dd t  +& \int_{\RRD} \varphi(0,\cdot) \dd \mu_0 \\
&=\int_0^{\infty} \int_{\RRD}w(t,x)\varphi(t,x)\dd \mu_t(x) \dd t,\\
\end{split}
\end{equation}
holds for all test functions $\varphi \in \CC_c^1([0,\infty) \times \RRD)$.
\end{definition}

We introduce a perturbation to the velocity field $b$ as follows
\begin{equation}\label{def:pert}
 b^h(t, x) {\cred \colonequals} b(t, x) + h \cdot b_1(t, x),
\end{equation}
where $\left(t \mapsto b(t, \cdot) \right), \left(t \mapsto b_1(t, \cdot) \right) \in \CC^0(\left[ 0, +\infty)\right.; \CCA(\RRD))$ {\cred and $h\in \mathbb{R}$, close to 0}. 

The perturbed problem corresponding to (\ref{def:prob_init}) has the form
\begin{equation}\label{def:prob_init2}
\left\{ \begin{array}{l}
\partial_t\mu_t^h + \textrm{div}_x\left( b^h( t, x)\mu_t^h \right) = w(t,x)\mu_t^h \quad \textrm{in }(\CC_c^1([0,\infty) \times \RRD))^{\ast}, \\
\mu_{t=0}^h= \mu_0 \in \mathcal{P}(\RRD). 
\end{array} \right. 
\end{equation}

Notice that the initial conditions in (\ref{def:prob_init}) and (\ref{def:prob_init2}) are the same ($\mu_{t=0}= \mu_{t=0}^h= \mu_0$).  For the purpose of further considerations, without loss of generality, we may assume that $h\in [-\frac{1}{2}, \frac{1}{2}]$.

Before stating the main result, we need to define an appropriate Banach space. First recall that the H\"older space $\CCA(\RRD)$ is a Banach space with the norm 
\begin{equation}
\|f\|_{\CCA(\RRD)} \colonequals  \sup_{x\in\RRD}|f(x)| +\ \sup_{x\in\RRD} |\nabla_x f(x)| \ +\
\sup_{\substack{x_1 \neq x_2\\ x_1, x_2 \in \RRD}}
 \frac{|\nabla_x f(x_1) - \nabla_x f(x_2)|}{|x_1-x_2|^\alpha}.
\end{equation}

The space of Radon measures $\mathcal{M}(\RRD)$ inherits the dual norm {\cred  of} $(\CCA(\RRD))^*$ by means of~embedding the former into the latter, where a measure is identified with the functional defined by  integration against the measure. 
 {\cred  Throughout} we identify the former with the subspace of~$(\CCA(\RRD))^*$. Let then
\begin{equation}\label{def:defZ} 
Z\colonequals \overline{{\mathcal{M}}(\RRD)}^{(\CCA(\RRD))^\ast}{\cred ,}
\end{equation}
which is a Banach space equipped with the dual norm $\Vert \cdot \Vert_{(\CCA(\RRD))^*}$.

We show in Proposition \ref{prop:Z_isom} that such defined $Z$ is a predual space of $\CCA(\RRD)$: $Z^*$ is linearly isomorphic to $\CCA(\RRD)$.

The following theorem is the main result of this paper.

\begin{theorem}\label{th:main}
{\cred Assume that} $\left(t \mapsto b(t, \cdot) \right){\textrm and } \left(t \mapsto b_1(t, \cdot) \right)\in \CC^0\big(\left[ 0, +\infty)\right.; \CCA(\RRD)\big)$, and  $w(t,x)\in \CCA([0, \infty)\times \RRD)$.
Let $\mu_t^h$ be {\cred the} weak solution to problem (\ref{def:prob_init2}) with velocity field defined by (\ref{def:pert}). Then the mapping
\[ [-\frac{1}{2}, \frac{1}{2}] \ni h \mapsto \mu_t^{h} \in \mathcal{P}(\RRD)\]
 is differentiable in $Z$, i.e. $
\partial_h \mu_t^{h} \in Z.
$
\end{theorem}

Classically, the analysis of structured population models was carried out in Lipschitz setting \cite{Webb:1985,Thieme:2003}. This approach is appropriate for considering the densities of populations. However, it does not allow to work with less regular distributions used in applications, like Dirac mass. Firstly, we would like to argue why this result cannot be obtained in the space $W^{1,\infty}$ with the~{\it flat metric} (called also {\it bounded Lipschitz distance}) -- what is a natural setting to consider transport equation in the space of bounded Radon measures \cite{Piccoli:2016,Piccoli:2017,jabl:2014,Crippa:2013,Czochra:2010_2}. 

Recall that the flat metric is defined as follows
\[\rho_F(\mu, \nu):= \sup_{f \in W^{1,\infty}, \Vert f \Vert_{W^{1, \infty}}\leq 1}\left\lbrace \int_{\RRD} f\dd(\mu-\nu)\right\rbrace.\]
It is worth recalling that the {\it generalized Wasserstein distance} coincides with the flat metric \cite{Piccoli:2016}.

Now, we recall a counterexample presented in \cite{Skrz:2018} for differentiability in the mentioned setting. Consider a perturbed transport equation for one dimensional $x$ on $\RR$
\begin{equation}\label{counter_example}
\left\{ \begin{array}{l}
\partial_t \mu_t^h + \partial_x((1+h)\mu_t^h)=0,\\ \mu_0^h=\delta_0.
\end{array}\right.
\end{equation}
It can be easily checked that $\mu_t^h=\delta_{(1+h)t}$ is a measure solution to (\ref{counter_example}). Note that the map $h\mapsto \mu_t^h$ is Lipschitz continuous
\[\rho_F(\mu_t^h,\mu_t^{h'})=\rho_F(\delta_{(1+h)t}, \delta_{(1+h')t})\leq |h-h'|t.\]
However,  it is not differentiable for the flat metric.
If $\frac{\mu_t^h-\mu_t^0}{h}$ were convergent, it would satisfy Cauchy condition with respect to the flat metric. We compute
\begin{equation*}
\int_{\RR}f(x)\left(\frac{\dd\mu_t^{h_1}(x)-\dd\mu_t^0}{h_1} - \frac{\dd\mu_t^{h_2}(x)-\dd\mu_t^0}{h_2} \right)=\frac{f((1+h_1)t)-f(t)}{h_1}-\frac{f((1+h_2)t)-f(t)}{h_2}.
\end{equation*}

If we choose a test function from $W^{1,\infty}(\RR)$ such that
\[f(x)=
\begin{cases}
|x-t|-1, & \textrm{if } |x-t| \leq 1,\\
0,	& \textrm{if } |x-t|>1,
\end{cases}\]
then for $h_1>0$ and $h_2<0$, we get $p_F\left(\frac{\mu_t^{h_1}-\mu_t^0}{h_1}, \frac{\mu_t^{h_2}-\mu_t^0}{h_2} \right) \geq 2t$.
Thus $\frac{\mu_t^h-\mu_t^0}{h}$ does not converge. 
That is why we need a space with test functions a little bit more regular than $W^{1,\infty}(\RR)$.

Theorem \ref{th:main}, differentiability with respect to perturbing parameter, is required for various applications. One that we shall discuss in this paper is application to~optimal control theory. 

As additional results we have further characterizations of the Banach space $Z$, {\cred p}resented in Section \ref{charZ}. First, $Z$ is separable as the span of Dirac measures at rational points is a~dense {\cred countable} subset of $Z$. Moreover we have {\cred that} $Z^*$  is linearly isomorphic to $\CCA(\RRD)$.

The outline of the paper is as follows. Section \ref{section:prelim} is devoted to preparing the necessary background in functional analysis. The proof of Theorem \ref{th:main} is treated in Section \ref{section:main_proof}. In Section \ref{section:opti} by discuss possible applications of the result of this paper. Characterization of the space $Z$ is presented in Section \ref{charZ}.

 
\section{Preliminaries}\label{section:prelim}
 
The characteristic system associated to equation \eqref{def:prob_init},  has the following form
\begin{equation}\label{def:diff_init}
\left\{ \begin{array}{l}
{\dot X_b}(t,y)=b\left(t, X_b(t,y)\right),\\
X_b(t_0, y)=y \in \RRD,
\end{array} \right.
\end{equation}
where $\left(t \mapsto b(t, \cdot)\right) \in \CC^0(\left[ 0, +\infty)\right.; \CCA(\RRD))$.

A solution to (\ref{def:diff_init}), $X_b$ is called a {\it flow map}. Note that the flow maps are defined for all $t \in \mathbb{R}$ and thus $y \mapsto X_b(t, y)$ is a one-parameter group of diffeomorphisms on $\RRD$ (dependent on the~variable~$b$). 

\begin{remark} The requirement $(t \mapsto b(t, \cdot)) \in \CC^0\left(\left[ 0, +\infty)\right.; \mathcal{C}^1(\RRD)\right)$ is sufficient to conclude that $y \mapsto X_b(t, y)$ is   a diffeomorphism. Higher regularity is needed when we estimate remainder terms of a Taylor expansion in the final proof of Theorem \ref{th:main} ({\cred see e.g. equation (\ref{eqn:holder_need})}). \end{remark}

Now we define the {\it push-forward operator} \cite{Ambrosio:2008}. If $Y_1$, $Y_2$ are separable metric spaces, $\mu \in \mathcal{P}(Y_1)$, and $r\colon  Y_1 \to Y_2$ is a $\mu$-measurable map, we denote by $\mu \mapsto  r\# \mu \in \mathcal{P}(Y_2)$ the~push-forward of $\mu$ through $r$, defined by
\[r\# \mu(\textrm{B})\colonequals  \mu(r^{-1}(\textrm{B})), \qquad \textrm{ for all } \textrm{B}\in \mathcal{B}(Y_2) 
.\]

The following lemma guarantees that a weak solution $\mu_t$ is probability measure.

\begin{lemma}[A representation formula for the non-homogenous continuity equation \cite{MANIGLIA2007601}]\label{lem:repr_form}
 Let $b(t, y)$ be a Borel velocity field in $L^1([0,T]; W^{1, \infty}(\RRD))$, $w(t,x)$ a Borel bounded and locally Lipschitz continuous (w.r.t. the space variable) scalar function and $\mu_0\in\mathcal{P}(\RRD)$. Then there exists a unique $\mu_t$, narrowly continuous family of Borel finite positive measures solving (in the distributional sense) the initial value problem (\ref{def:prob_init}) and it is given by the explicit formula
 \[\mu_t=X_b(t,\cdot)\#( e^{\int_0^t w(s, X_b(s,\cdot))\dd s} \cdot \mu_0), \qquad  \textrm{ for all } t \in[0,T].\]
 
\end{lemma}


\begin{remark} Since in our case $(t \mapsto b(t, \cdot)) \in \CC^0\left(\left[ 0, +\infty)\right.; \CCA(\RRD)\right)$ then $b$ is globally Lipschitz and thus the solution $X_t$ is global. Also $w(t,x)$ satisfies the assumption in Lemma \ref{lem:repr_form}. Thus we conclude that (\ref{def:prob_init}) has a unique weak solution $t\mapsto \mu_t$, that  is defined for all $t$. \end{remark}
 
Since the representation formula could be generalized for the case when $\mu$ is a non-negative measure $\mathcal{M}^+(\RRD)$ we can also consider non-positive measures as an initial condition.


\section{Proof of main result -- Theorem \ref{th:main}}\label{section:main_proof}
By definition $Z= \overline{\textrm{span}\lbrace \delta_x\colon  x\in \RRD\rbrace}^{(\CCA(\RRD))^*}$ is a subspace of $(\CCA(\RRD))^*$. The space $Z$ inherits {\cred the} norm of $(\CCA(\RRD))^*$. Since $Z$ is complete, it is enough to show that proper sequence of differential quotient is a Cauchy sequence. 

The analogue of (\ref{def:diff_init}) for the system associated to perturbed equation \eqref{def:prob_init2} with velocity field defined by (\ref{def:pert}), where $\left(t \mapsto b(t, \cdot)\right), \left(t \mapsto b_1(t, \cdot)\right) \in \CC^0\left(\left[ 0, +\infty)\right.; \CCA(\RRD)\right)$, has the form
\begin{equation}\label{def:diff_pert}
\left\{ \begin{array}{l}
{\dot X}_{h}(t,y)=\left(b+b_1 h\right)\left(t, X_{h}(t,y)\right),\\
X_{h}(t_0, y)=y \in \RRD.
\end{array} \right.
\end{equation}

As before, $y \mapsto X_{h}(t, y)$ is a diffeomorphism. To underline the dependence of $X_{h}(t, x)$ on the~parameter $h$ from now on we will use the notation $X(t, y; h)\colonequals X_{h}(t, y)$.

\begin{lemma}{\label{claim:esti_X}}
Let $\left(t \mapsto b(t, \cdot)\right), \left(t \mapsto b_1(t, \cdot)\right) \in \CC^0(\left[ 0, +\infty)\right.; \CCA(\RRD))$. Then for all $(t, y)\in [0,+ \infty) \times \RRD$ the mapping $\left(h \mapsto X (t, y; h)\right) \in \CCA([-\frac{1}{2}, \frac{1}{2}])$. 

\end{lemma}

The proof goes in a similar way as the proof of higher order differentiability ($\mathcal{C}^k$, where $k\in\mathbb{N}$) of the solution with respect to parameters, which can be found in the book \cite{ODE:1982} p. 100.

We are in the position to prove the main result.

\begin{proof}[Proof of Theorem \ref{th:main}] 
Consider the weak solution $\mu_t$ to system (\ref{def:prob_init}) (where $h=0$) and $\mu_t^{h_1}$, $\mu_t^{h_2}$ ($h_1\neq h_2$, $h_{1,2}\neq 0$) to system defined by (\ref{def:prob_init2}).  They are unique and defined for all $t\in [0, \infty)$, according to Lemma \ref{lem:repr_form}.

Notice that for every $\lambda\in\mathbb{R}$, $\frac{\mu_t^{h+\lambda}-\mu_t^h}{ \lambda}\in \mathcal{M}(\RRD) \subseteq Z$, which is a complete space. 
 First we show differentiability at $h=0$. Differentiability at other $h$ follows from this result (see end of proof).
 
 For the first part it suffices to show that
 \[I_{{\cblue{h_1},{h_2}}} \colonequals
\left\Vert  \frac{\mu_t^{h_1}-\mu_t}{h_1}\right. - \left.\frac{\mu_t^{h_2}-\mu_t}{h_2}\right\Vert_{(\CCA (\RRD))^*}\]
 can be made arbitrary small, when $h_1$ and $h_2$ are sufficiently close to 0. Then for any sequence $h_n\to 0$, $\frac{\mu_t^{h_n}-\mu_t}{h_n}$ is a Cauchy sequence in $(\CCA(\RRD))^*$. Hence, converges to a limit that is the same for each sequence $(h_n)$ such that $h_n\to 0$.
\begin{equation}\label{eqn:final_Cauchy}
\begin{split}
{\cblue I_{{\cblue{h_1},{h_2}}}}& =
\sup_{\Vert \psi \Vert_{\CCA} \leq 1} \left|\int_{\RRD} \psi d\left( \frac{\mu_t^{h_1}-\mu_t}{h_1} - \frac{\mu_t^{h_2}-\mu_t}{h_2} \right)\right|
\\
&=
\sup_{\Vert \psi \Vert_{\CCA} \leq 1} \left| \int_{\RRD} \psi \frac{\dd \mu_t^{h_1}}{h_1} - \int_{\RRD} \psi \frac{\dd \mu_t}{h_1} - \int_{\RRD} \psi \frac{\dd \mu_t^{h_2}}{h_2} + \int_{\RRD} \psi \frac{\dd \mu_t}{h_2}\right|
\end{split}
\end{equation}

First we use representation formula (Lemma \ref{lem:repr_form}) and the fact that $y \mapsto X(t, y; h)$ is a diffeomorphism. Introduce for convenience $\overline{w}(s, y; h) \colonequals w(s, X_b(s, y; h))$.
\begin{equation*}
\begin{split}
I_{{\cblue{h_1},{h_2}}}=&\sup_{\Vert \psi \Vert_{\CCA}\leq 1}
\left|
\int_{\RRD} \psi(X(t,y; h_1)) \Eehj \frac{\dd \mu_0}{h_1}
-
\int_{\RRD} \psi(X(t,y; 0))\Eehz\frac{\dd \mu_0}{h_1}\right.\\
& \qquad \qquad 
-\left.
\int_{\RRD} \psi(X(t,y;h_2))\Eehd\frac{\dd \mu_0}{h_2}
+
\int_{\RRD} \psi(X(t, y; 0))\Eehz\frac{\dd \mu_0}{h_2}
\right|\\
&= \sup_{\Vert \psi \Vert_{\CCA}\leq 1} \underbrace{\Bigg|\int_{\RRD} \left(\psi(X(t,y; h_1))-\psi(X(t,y; 0))\right)\Eehj \frac{\dd \mu_0}{h_1}}_{{\cblue I_{h_1}^{(1)}}}\\
&\qquad \qquad - \underbrace{\int_{\RRD} \left(\psi(X(t,y; h_2))-\psi(X(t,y; 0))\right)\Eehd \frac{\dd \mu_0}{h_2}}_{{\cblue I_{h_2}^{(1)}}}\\
&\qquad \qquad - \underbrace{\int_{\RRD} \left( \Eehz -\Eehj\right)\psi(X(t,y; 0)) \frac{\dd \mu_0}{h_1}}_{{\cblue I_{h_1}^{(2)}}}
\\
&\qquad \qquad + \underbrace{\int_{\RRD} \left( \Eehz -\Eehd\right)\psi(X(t,y; 0)) \frac{\dd \mu_0}{h_2}\Bigg|}_{{\cblue I_{h_2}^{(2)}}} 
\end{split}
\end{equation*}

Let us consider $|{\cblue I_{h_1}^{(1)}}-{\cblue I_{h_2}^{(1)}}|$ and $|{\cblue I_{h_1}^{(2)}}-{\cblue I_{h_2}^{(2)}}|$ separately.

In ${\cblue I_{h_1}^{(2)}}-{\cblue I_{h_2}^{(2)}}$ expand $\Eehj$ and $\Eehd$ into Taylor series around $h=0$
\begin{equation}\label{eqn:holder_need}
\begin{split}
|{\cblue I_{h_1}^{(2)}}-&{\cblue I_{h_2}^{(2)}}| = \Bigg|\int_{\RRD} \psi(X(t,y; 0))\Big[\Eehz-\Eehz\\
&\qquad \qquad -h_1\Eehz \partial_h\Big(\int_0^t \overline{w}(s, y;h) \dd s\Big)\Big|_{h=0}-\mathcal{O}(|h_1|^{1+\alpha})\Big]\frac{\dd \mu_0}{h_1}
\\
&\qquad-\int_{\RRD} \psi(X(t,y; 0))\Big[\Eehz-\Eehz
\\
&\qquad \qquad - h_2\Eehz \partial_h\Big(\int_0^t \overline{w}(s, y;h) \dd s\Big)\Big|_{h=0}-\mathcal{O}(|h_2|^{1+\alpha})\Big]\frac{\dd \mu_0}{h_2}\Bigg|
\\
=&\Bigg|
\psi(X(t,y; 0))\int_{\RRD} \Big[
-\Eehz \partial_h\Big(\int_0^t \overline{w}(s, y;h) \dd s\Big)\Big|_{h=0}\\
& + 
\Eehz \partial_h \Big(\int_0^t \overline{w}(s, y;h) \dd s\Big)\Big|_{h=0}-\frac{\mathcal{O}(|h_1|^{1+\alpha})}{h_1}+ \frac{\mathcal{O}(|h_2|^{1+\alpha})}{h_2}\Big]\dd \mu_0\Bigg|
\end{split}
\end{equation}
\begin{equation*}
\begin{split}
\leq& \Bigg|\psi(X(t,y; 0))\int_{\RRD} \left(-\frac{\mathcal{O}(|h_1|^{1+\alpha})}{h_1}+ \frac{\mathcal{O}(|h_2|^{1+\alpha})}{h_2}\right)\dd \mu_0\Bigg|\\
 \leq & \Big|c\int_{\RRD} \left(-\frac{\mathcal{O}(|h_1|^{1+\alpha})}{h_1}+ \frac{\mathcal{O}(|h_2|^{1+\alpha})}{h_2}\right)\dd \mu_0\Big|.
\end{split}
\end{equation*}


We now take into consideration $|{\cblue I_{h_1}^{(1)}}-{\cblue I_{h_2}^{(1)}}|$. Because $\psi\in\CCA(\RRD)$, one has
\begin{equation}\label{taylor_ex}
\psi(x)=\psi(x_0)+\nabla_x \psi(x_0)(x-x_0) + R(x,x_0),\tag{$\star$}
\end{equation}
with $|R(x, x_0)| \leq C|\nabla_x\psi|_\alpha \Vert x-x_0\Vert^{1+\alpha}$, where $|\nabla_x \psi|_\alpha$ is an $\alpha$-H\"older constant. Thus, expand $\psi(X(t,y; h_1))
$ and $\psi(X(t,y; h_2))$ into Taylor series around $X(t,y; 0)$
\begin{equation*}
\begin{split}
|I_{h_1}^{(1)}-I_{h_2}^{(1)}|&=
\Big|\int_{\RRD} \left[ \psi(X(t,y; 0))
+
\nabla_x\psi(X(t, y; h))|_{h=0} \cdot (X(t, y; h_1) - X(t, y; 0))\right. \\
& \qquad \qquad  + \left. \mathcal{O}\left(\left|X(t, y; h_1)-X(t, y; 0)\right|^{1+\alpha}\right)
- \psi(X(t, y; 0))\right] \Eehj\frac{\dd \mu_0}{h_1}
\\
& \qquad - 
\int_{\RRD}  \left[ \psi(X(t, y;0))
+
\nabla_x\psi(X(t, y; h))|_{h=0}(X(t, y; h_2) - X(t, y; 0))\right.\\
& \qquad \qquad  + \left. \mathcal{O}\left(\left|X(t, y; h_2)-X(t, y; 0)\right|^{1+\alpha}\right)
- \phiz\right]\Eehd \frac{\dd \mu_0}{h_2}\Big|.
\end{split}
\end{equation*}
Expanding $X(t, y; h_2)$ and $X(t, y; h_1)$ around $h=0$, by Lemma \ref{claim:esti_X} and expansion similar to (\ref{taylor_ex}) for $h \mapsto X(t,y;h)$ we obtain
\begin{equation*}
\begin{split}
|{\cblue I_{h_1}^{(1)}}-&{\cblue I_{h_2}^{(1)}}|\\
=&\Bigg|
\int_{\RRD} \Big[
\nabla_x\psi(X(t,y; h))|_{h=0}\Big(X(t,y; 0)+ h_1\partial_h X(t, y;h)\Big|_{h=0} + \mathcal{O}(|h_1|^{1+\alpha}) - X(t,y; 0)\Big)\\
&\qquad \qquad
 +
 \mathcal{O}\Big(\left|X(t,y; h_1)-X(t,y; 0)\right|^{1+\alpha}\Big)\Big] \Eehj\frac{\dd \mu_0}{h_1}\\
&-\int_{\RRD} \Big[
\nabla_x\psi(X(t,y; h))|_{h=0}\Big(X(t,y; 0)+ h_2\partial_h X(t, y;h)\Big|_{h=0} + \mathcal{O}(|h_2|^{1+\alpha})- X(t,y; 0)\Big)\\
&\qquad \qquad
+
 \mathcal{O}\Big(\left|X(t,y; h_2)-X(t,y; 0)\right|^{1+\alpha}\Big)\Big] \Eehd\frac{\dd \mu_0}{h_2}\Bigg|.
\end{split}
\end{equation*}
Since the remainder term $\mathcal{O} \Big(\left|X(t, y; h)-X(t, y; 0)\right|^{1+\alpha}\Big) \leq c|h|^{1+\alpha}$ for all $h\in [-\frac{1}{2}, \frac{1}{2}]$, we can further estimate
\begin{equation*}
\begin{split}
|{\cblue I_{h_1}^{(1)}}-&{\cblue I_{h_2}^{(1)}}|\\
\leq &\Big|\int_{\RRD} \left[
\nabla_x\psi(X(t,y; h))|_{h=0}\Big(\partial_h X(t, y;h)\Big|_{h=0}+\mathcal{O}(|h_1|^\alpha)\Big)+ \mathcal{O}(|h_1|^\alpha)
\right]\cdot \Eehj \dd \mu_0\\
&
-\int_{\RRD} \left[
\nabla_x\psi(X(t,y; h))|_{h=0}\Big(\partial_h X(t, y;h)\Big|_{h=0}+\mathcal{O}(|h_2|^\alpha)\Big)+ \mathcal{O}(|h_2|^\alpha)
\right]\cdot \Eehd \dd \mu_0\Big|.
\end{split}
\end{equation*}

We consider function $\psi\in \CCA$, with $\Vert \psi \Vert_{\CCA}\leq 1$. Hence we can further estimate\\ $\nabla_x\psi(X(t,y; h))|_{h=0}\leq 1$. This yields
\begin{equation*}
\begin{split}
|{\cblue I_{h_1}^{(1)}}-{\cblue I_{h_2}^{(1)}}| \leq& \Big|\int_{\RRD}\Big[\partial_h X(t, y;h)\Big|_{h=0}\left( \Eehj-\Eehd \right) \\
&\qquad \qquad + \left( \mathcal{O}_1(|h_1|^\alpha)\Eehj -\mathcal{O}_1(|h_2|^\alpha)\Eehd \right)
\Big] \dd \mu_0\Big|.
\end{split}
\end{equation*}
To summarize estimations of $|{\cblue I_{h_1}^{(1)}}-{\cblue I_{h_2}^{(1)}}|$:
\begin{itemize}
\item $\partial_h X(t, y;h)\Big|_{h=0}$ is just finite number (Lemma \ref{claim:esti_X}),
\item $\left( \Eehj-\Eehd \right)$ can be estimated by $c|h_1-h_2|$ (argumentation is similar as in estimations of $|{\cblue I_{h_1}^{(2)}}-{\cblue I_{h_2}^{(2)}}|$),
\item $\left( \mathcal{O}(|h_1|^\alpha)\Eehj -\mathcal{O}(|h_2|^\alpha)\Eehd \right)$ is going to zero when $h_1\to 0$ and $h_2 \to 0$.

\end{itemize}

{\cblue Thus $I_{{h_1},{h_2}}$ can be made arbitrarily small when $h_1$ and $h_2$ are sufficiently close to 0.}
Therefore we have shown that $\frac{\mu_t^{h+\lambda_n}-\mu_t^h}{\lambda_n}$ is a Cauchy sequence for every $\lambda_n\to 0$ in $(\CCA(\RRD))^*$ for $h=0$, {\cblue with the same limit. Hence} 
$\mu_t^h$ is differentiable with respect to parameter $h$ at $h=0$.

The same argumentation works for $h\neq 0$. Let us consider a sequence $\frac{\mu_t^{h+\lambda_n}-\mu_t^h}{\lambda_n}$, where $\lambda_n \to 0$ and $h\neq 0$. 
 By definition of perturbation (\ref{def:pert}), i.e. $b^h\colonequals b+hb_1$, the solution $\mu_t^{h+\lambda_n}$ for velocity field
$$b^{h+\lambda_n}=b+hb_1+\lambda_nb_1 =: \overline{b} +\lambda_n b_1$$
and initial condition $\mu_0^{h+\lambda_n}=\mu_0$ is equal (by Lemma \ref{lem:repr_form}) to the solution $\overline{\mu}_t^{\lambda_n}$ with velocity field ${\cblue \overline{b}}+\lambda_n b_1$ and initial condition $\mu_0$. A similar statement holds for the $\mu_t^h$ and the solution $\overline{\mu}_t^0$ of (\ref{def:prob_init}) with velocity field $\overline{b}$. Thus
\[\frac{\mu_t^{h+\lambda_n}-\mu_t^h}{\lambda_n}=\frac{\overline{\mu}_t^{\lambda_n}-\mu_0}{\lambda_n}\]
and the latter sequence converges in $Z$ as $h\to \infty$, by the first {\cblue part} of {\cblue the} proof.
\end{proof}

\section{Application to Optimal Control}\label{section:opti}
The results discussed above can be applied in optimal control theory.  The list of references on optimal problems concerning transport equation is steadily growing \cite{Gango:2013,Bongini:2017,Fornaster:2017,Degong:2017,Herty:2018a,Bonnet:2018a}.

There are two main approaches to solve optimal control problem when the solution is not differentiable with respect to the control parameter. The first one is just to use {\it non-smooth analysis}. The second one is to strengthen assumptions
for the problem to
provide the solution will be differentiable
, and then use {\it smooth analysis} -- for which there are developed  significantly more tools, and which are less numerically complex than non-smooth methods.

In \cite{Bonnet:2018a} the authors consider the following optimal control problem
\begin{equation}\label{prob:Bonnet_opt}
\left\{ \begin{array}{l}
\max_{u\in\mathcal{U}} \left[\int_0^T L(\mu_t, u(t,\cdot))\dd t + \psi(\mu_T) \right],
\\
\\
\left\{ \begin{array}{l}
\partial_t \mu_t + \div_x\Big((v(t,\mu_t,\cdot) +u(t,\cdot))\mu_t\Big)=0,\\ \mu_{t=0}=\mu_0\in\mathcal{P}_c(\RRD),
\end{array}\right.
\end{array}\right.
\end{equation}
where $\mathcal{P}_c(\RRD)$ is the subset of $\mathcal{P}(\RRD)$ of Borel probability measures with compact support.
The function $L$ can be interpreted as income dependent on the level of sales (which is described by measure $\mu_t$) and a situation on the market, $u(t)$ . In this optimal control problem, we want to maximize the total income in the period $[0,T]$. Function $\psi(\mu_T)$ describes the income in a terminal time $T$.

Notice that the period of time is finite and $v(t,\mu_t,x) +u(t,x)$ 
corresponds to the velocity field $b(t,x)$ in our transport equation. Contrary to the problem~(\ref{def:prob_init}) the authors consider the term $v(t,\mu_t,x)$ which depends on the solution. The studies on such non-linear problems will appear in \cite{lycz:2019}.
Nevertheless, briefly speaking, the assumptions for coefficients of~(\ref{prob:Bonnet_opt}) are weaker than the ones for~(\ref{def:prob_init}). In particular the velocity field $v(t,\mu_t,x) +u(t,x)$ is not differentiable with respect to perturbation in $u(t,x)$, it just satisfies Lipschitz condition. 

The authors formulated a new Pontryagin Maximum Principle in the language of subdifferential calculus in Wasserstein spaces.

Below we would like to present the second approach. We want to argue how differentiability of velocity field with respect to perturbing parameter in problem~(\ref{def:prob_init2}) can be applied in optimal control.

{\cred In control theory, the control is based on observation of the state of the system at each or some finite points: $u(t)\colonequals \phi(\mu_t^h)$. The state $\mu_t^h$ is in $\mathcal{M}(\RRD)\subset Z$. Thus, a reasonable class of differentiable observation function $\phi$ is provided by the composition of a continuous linear functional on $Z$ and $f\in \mathcal{C}^1(\mathbb{R}, \mathbb{R})$. In Proposition \ref{prop:Z_isom} we show that every continuous linear functional on $Z$ is represented essentially by integration with respect to a $\CCA(\RRD)$-function -- denote it here by $K$.}

{\cred Thus,} aiming at optimal control of the solution to (\ref{def:prob_init2}), where $h$ is a control parameter attaining values in $\mathbb{R}$, {\cred we start by considering functionals of the form}
\begin{equation}\label{def:func_opti}
\gamma(\mu^h)\colonequals \widehat{\gamma} \left(\int_\RRD K(x)\dd \mu^h(x)\right),
\end{equation}
where $\widehat{\gamma}$ is a $\mathcal{C}^1$-function {\cred and $K\in \CCA(\RRD)$}.

{\cred The meaning is essentially the following:} the integral operator $\int_{\RRD} K(x)\dd \mu(x)$ is well-defined for $\mu$ being a measure and necessary not every element from the space $Z$ is measure. Following lemma provides extension of the domain to whole space $Z$.

\begin{lemma}[Extension Theorem]\cite[Theorem 2.1]{amman_escher}
Suppose $X$ and $Y$ are metric spaces, and $Y$ is
complete. Also suppose $X_1$ is a dense subset of $X$, and $f\colon  X_1 \to Y$ is uniformly
continuous. Then $f$ has a uniquely determined extension $\overline{f}\colon  X \to Y$ given by
\[ \overline{f}(x) = \lim_{x_1\to x, x_1\in X_1} f(x_1), \qquad {\textrm{for }} x \in X,\]
and $\overline{f}$ is also uniformly continuous.
\end{lemma}

In our case {\cred the} operator $\int_{\RRD} K(x)\dd \mu(x)$ is of course well-defined for any $\mu\in \mathcal{M}(\RRD)$ {\cred and} it {\cred can} be uniquely extended to $Z=\overline{\mathcal{M}(\RRD)}^{(\CCA (\RRD))^*}$ ($\textrm{span}\{ \delta_x\colon  x\in \RRD\}$ is dense subset of $Z$, Proposition \ref{prop:separable}). Denote this uniquely determined extension by
\[\langle K(\cdot), \mu\rangle_{\CCA(\RRD), Z},\]
where $\langle \cdot, \cdot \rangle$ is dual pair. Thus the functional corresponding to (\ref{def:func_opti}) has the form
\begin{equation}\label{def:dual_opti}
\overline{\gamma}(\mu^h)=\widehat{\gamma}\left(\langle K, \mu\rangle_{\CCA(\RRD), Z}\right).
\end{equation}

Now consider the problem 
\begin{equation}\label{prob_opt}
\min_{h\in\RR}\overline{\gamma}(\mu^h).
\end{equation}
 That is, we wish to find an $h^*\in \mathbb{R}$ such that $\overline{\gamma}(\mu^{h^*})\leq \overline{\gamma}(\mu^h)$ for all $h\in \mathbb{R}$. 


A necessary condition for $\mu^{h^*}$ realizing a minimum is that the gradient of the function $\overline{\gamma}$ is zero at $\mu^{h^*}$
\begin{equation}\label{cond:min}
\left.\partial_h \overline{\gamma}(\mu^{h}) \right|_{h=h^*} =
\widehat{\gamma}'\big|_{\left\langle K,  \mu^h \right\rangle} \cdot \left\langle K,  \left. \partial_h\mu^{h}\right|_{h=h^*} \right\rangle_{\CCA(\RRD), Z}=0.
\end{equation}

For this condition to be satisfied {\cred it is} necessary {\cred that} {\cred $h \mapsto$} $\overline{\gamma}(\mu^h)\in \mathcal{C}^1(Z, \RR)$. This is guaranteed by the following lemma when combined with the differentiability of $\mu^h$ with respect to $h$ (Theorem \ref{th:main}).

\begin{lemma}\label{lem:opti}
If $K(x)\in \CCA(\RRD)$ and $\widehat{\gamma}\in \mathcal{C}^1( \RR)$ then $\overline{\gamma}$ defined by (\ref{def:dual_opti}) is $\mathcal{C}^1(Z, \RR)$.
\end{lemma}

\begin{proof}
What we want to show is that if $K\in \CCA(\RRD)$ then {\cred the} functional $\mu \mapsto \langle K, \mu\rangle_{\CCA(\RRD), Z}$ is linear and bounded on $Z$. Then $\widehat{\gamma}(\langle K, \mu\rangle_{\CCA(\RRD), Z})\in \mathcal{C}^1(Z, \RR)$, as a composition of $\mathcal{C}^1$-function and a bounded linear functional.

Linearity of $\mu\mapsto\langle K, \mu\rangle_{\CCA(\RRD), Z}$ is clear. Following holds
\begin{equation*}
\begin{split}
\left|\langle K, \mu\rangle_{\CCA(\RRD), Z}\right| \leq \Vert K\Vert_{\CCA(\RRD)} \cdot \Vert \mu \Vert_{(\CCA(\RRD))^*}\leq C\Vert \mu \Vert_{(\CCA(\RRD))^*},
\end{split}
\end{equation*}
where constant $C=\sup_{x\in\RRD}|K(x)|+\sup_{x\in\RRD}|\nabla K(x)|+ \sup_{x, y \in \RRD x\neq y} \frac{|\nabla K(x)-\nabla K(y)|}{|x-y|^\alpha}$.

Thus the functional $\mu\mapsto\langle K, \mu\rangle_{\CCA(\RRD), Z}$ is bounded. We conclude that $\overline{\gamma} \in \mathcal{C}^1(Z, \RR)$.
\end{proof}

Of course, there are many optimization methods which do not depend on finding derivative analytically and then setting it to zero. When a functional $\overline{\gamma}(\mu^h)$ is differentiable with respect to $h$, an optimization problem (\ref{prob_opt}) can be solved with gradient-based analytical methods or through numerical methods such as the steepest descent. When $\overline{\gamma}(\mu^h)$ is not differentiable, the above-mentioned methods cannot be applied, the problem becomes more complex numerically.  And for differentiability of $\overline{\gamma}(\mu^h)$ necessary is differentiability of $\mu^h$, which is satisfied by Theorem \ref{th:main}.

\begin{remark} If $\widehat{\gamma}$ is convex then condition ({\ref{cond:min}}) is not only necessary but also sufficient for $\mu^{h^*}$ to realize a minimum.\end{remark}

\subsection{Further application}

{\cred
In \cite{Rosi:2016} authors consider optimization in the structured population model defined by
\begin{equation}\label{apl:str}
\left\{
                \begin{array}{l}
                  \partial_t\mu_t+ \partial_x\Big(b(t)(\mu_t, x)\mu_t\Big)+w(t)(\mu_t, x)\mu_t=0,\\
                  \big(b(t)(\mu_t, 0)\big)D_{\lambda}\mu_{t}(0)=\int_0^{\infty}\beta(t) (\mu_t, x)\dd \mu_t,\\
                  \mu_{t=0}=\mu_0,
                \end{array}
              \right.
\end{equation}
 where $t\in [0, \infty)$ and $x\in\mathbb{R}_{+}$ is a biological parameter, typically age or size. The
unknown $\mu_t$ is a time dependent, non-negative and finite Radon measure. The growth function $b$ and the~mortality rate $w$ are strictly positive, while the birth function $\beta$ is non-negative -- $b, w, \beta$ are Nemytskii operators.
By $D_{\lambda}\mu_{t}$ we denote the Radon–-Nikodym derivative of $\mu_t$ with respect to the Lebesgue
measure $\lambda$ computed at 0. The initial datum $\mu_0$ is a non-negative Radon measure.

\begin{remark}
The reason for analyzing solutions to structured population models in the space of~measures is as follows: typical experimental data are not continuous, they provide information on percentiles, i.e., the number of individuals in some intervals of the structural variable (like age). In the case of demography and epidemiology a number of births are typically used per years. 
\end{remark}

Aiming at the optimal control of the solution to (\ref{apl:str}), a control parameter $h$ is introduced (possibly time and/or state dependent), attaining values in a given set $\mathcal{H}$. Therefore,
we obtain:
\begin{equation}\label{apl:str_per}
\left\{
                \begin{array}{l}
                  \partial_t\mu_t^h+ \partial_x\big(b(t;h)(\mu_t^h, x)\mu_t^h\big)+w(t; h)(\mu_t^h, x)\mu_t^h=0,\\
                  \big(b(t; h)( \mu_t^h, 0)\big)D_{\lambda}\mu_{t}^h(0)=\int_0^{\infty}\beta(t; h)(\mu_t^h, x)\dd \mu_t,\\
                  \mu_{t=0}^h=\mu_0.
                \end{array}
              \right.
\end{equation}

The goal is to find minimum of a given functional 
\begin{equation}\label{eqn:functional}
\mathcal{J}(\mu_t^h)=\int_0^{\infty} j(t, \mu_t^h; h) \dd t,
\end{equation}
within a suitable function space i.e. to find an $h^*\in \mathcal{H}$ such that $\mathcal{J}(\mu^{h^*})\leq \mathcal{J}(\mu^h)$ for all $h\in \mathcal{H}$. 

Aiming at the optimal control problem in \cite{Rosi:2016} the Escalator Boxcar Train (EBT) algorithm is adapted (defined in \cite{jabl:2014}), i.e. an appropriate ODE system is used approximating the~original PDE model. Authors mention that solutions to conservation or balance laws typically depend in a Lipschitz continuous way on the initial datum as well as from the functions defining the equation. This does not allow the use of differential tools in the search for the~optimal control.

Since solution to the transport equation is differentiable with respect to parameter, mathematical tools applied to (\ref{eqn:functional}) can be extended by e.g. gradient methods.
}

\section{Characterization of the space $Z$}\label{charZ}
In this section we {\cred establish some further properties of} the space $Z$ defined by (\ref{def:defZ}). {\cred The identification of the dual space $Z^*$ in Proposition \ref{prop:Z_isom} is particularly interesting eg. in view of the~application to control theory, discussed in Section \ref{section:opti}.} By $\delta_x$ we denote the Dirac measure concentrated in $x$.  
\begin{prop}\label{prop:separable}
Let $Z$ be given by (\ref{def:defZ}). Then the set $\mathrm{span}\{\delta_x\colon  x\in\mathbb{Q}^d \}$ is dense in $Z$ with respect to {\cred the} $(\CCA(\RRD))^*$-topology, i.e.
\[Z=\overline{\mathrm{span}\{\delta_x\colon x\in\mathbb{Q}^d \}}^{(\CCA(\RRD))^*}.\]

Consequently, $Z$ is a separable space.
\end{prop}

\begin{proof}
 We want to show that for any measure $\mu\in \mathcal{M}(\RRD)$ there exists a sequence $\lbrace \mu_n \rbrace_{n\in \mathbb{N}} \in {\rm span}\{ \delta_x\colon  x \in \mathbb{Q}^d \}$ such that $\Vert\mu_n - \mu\Vert_{(\CCA(\RRD))^*} \to 0 $ as $n\to \infty$. 

We consider bounded Radon measures, thus
for any $\mu\in \mathcal{M}(\RRD)$  and for any $\varepsilon>0$ there exists $R_{\varepsilon}$  such that $|\mu|(\RRD\setminus \mathcal{B}(0, R_{\varepsilon}))\leq \frac{\varepsilon}{2}$. The closure of a ball $\mathcal{B}(0, R_{\varepsilon})$ in $\RRD$ as a compact set has finite cover $\lbrace{\mathcal{B}(g_i, \frac{\varepsilon}{4 \Vert\mu\Vert_{\textrm{TV}}})\rbrace}_{i=1}^{n(\varepsilon)}$, where $g_i\in\mathbb{Q}^d$.
Denote by $\mathcal{B}_i\colonequals \mathcal{B}(g_i, \frac{\varepsilon}{4 \Vert\mu\Vert_{\textrm{TV}}})$. Then define
 \begin{equation}\label{def:u}
 U_{i, \varepsilon}\colonequals \left(\mathcal{B}(0, R_\varepsilon) \cap \mathcal{B}_i\right) \backslash \cup_{j=1}^{i-1} \mathcal{B}_j
 \end{equation} are disjoint Borel sets and $\cup_{i=1}^{n(\varepsilon)} U_{i, \varepsilon} = \mathcal{B}(0, R_\varepsilon)$.  
Notice that $g_i$ (the center of $\mathcal{B}_i$) is not necessarily contained in $U_{i,\varepsilon}$. In case $g_i$ is not contained in $U_{i, \varepsilon}$ we take any other point of the ball $\mathcal{B}_i$ contained in $U_{i,\varepsilon}$, we will denote this point the same way, slightly abusing notation.

For any $\mu \in \mathcal{M}(\RRD)$ and any $\varepsilon>0$ we consider $\mu^{\varepsilon} = \sum_{i=1}^{n(\varepsilon)} \mu(U_{i, \varepsilon})\cdot \delta_{g_i}$ (linear combination of~Dirac deltas concentrated at points $g_i \in \mathbb{Q}^d$). Denote by $\widehat{\mu}\colonequals \left.\mu\right|_{\mathcal{B}(0,R_\varepsilon)}$ the measure restricted to $\mathcal{B}(0, R_\varepsilon)$. Then the following holds:
\begin{align*}
\Vert \mu^{\varepsilon} - \mu \Vert_{(\CCA(\RRD))^*} \leq \left\Vert \mu^\varepsilon - \widehat{\mu}\right\Vert_{(\CCA(\RRD))^*}+ \left\Vert \widehat{\mu} - \mu \right\Vert_{(\CCA(\RRD))^*}
\leq  \left\Vert \mu^\varepsilon - \widehat{\mu}\right\Vert_{(\CCA(\RRD))^*} + \frac{\varepsilon}{2}.
\end{align*}

We need to estimate the following
\begin{align*}
&\left\Vert \mu^\varepsilon - \widehat{\mu} \right\Vert_{(\CCA(\RRD))^*} \\
&= \sup \left\lbrace\int_{\RRD} f \dd\left(\mu^\varepsilon - \widehat{\mu}\right)\colon f\in\CCA(\RRD), \Vert f \Vert_{\CCA(\RRD)} \leq 1 \right\rbrace\\
& \leq \sup \left\lbrace\int_{\RRD} f \dd(\mu^\varepsilon - \widehat{\mu})\colon f \in \mathrm{Lip}(\RRD), \Vert f\Vert_\infty + \Vert \nabla f \Vert_\infty \leq 1 \right\rbrace\\
&=\sup \left\lbrace\int_\RRD f\dd \Big(\sum_{i=1}^{n(\varepsilon)}\left(\mu(U_{i, \varepsilon})\delta_{g_i}- \widehat{\mu}\right)\Big)\colon f \in \mathrm{Lip}(\RRD), \Vert f\Vert_\infty + \Vert \nabla f \Vert_\infty \leq 1 \right\rbrace\\
\end{align*}
\begin{align*}&= \sup \left\lbrace\sum_{i=1}^{n(\varepsilon)}\left(\int_{U_{i, \varepsilon}}f(g_i) \dd \mu- \int_{U_{i, \varepsilon}}f\dd \mu\right)\colon f \in \mathrm{Lip}(\RRD), \Vert f\Vert_\infty + \Vert \nabla f \Vert_\infty \leq 1 \right\rbrace\\
&= \sup \left\lbrace\sum_{i=1}^{n(\varepsilon)}\int_{U_{i, \varepsilon}}(f(g_i)-f)\dd \mu\colon f \in \mathrm{Lip}(\RRD), \Vert f\Vert_\infty + \Vert \nabla f \Vert_\infty \leq 1 \right\rbrace\\
&\leq \sup \left\lbrace\sum_{i=1}^{n(\varepsilon)}\int_{U_{i, \varepsilon}}|g_i-x| \dd|\mu|\colon f \in \mathrm{Lip}(\RRD), \Vert f\Vert_\infty + \Vert \nabla f \Vert_\infty \leq 1 \right\rbrace\\
&=\sum_{i=1}^{n(\varepsilon)}\int_{U_{i, \varepsilon}}|g_i-x| \dd|\mu|
\leq\sum_{i=1}^{n(\varepsilon)}\int_{U_{i, \varepsilon}} \frac{\varepsilon}{2\Vert\mu\Vert_{\textrm{TV}}}\dd|\mu|\\
& = \frac{\varepsilon}{2\Vert\mu\Vert_{\textrm{TV}}} \int_{\mathcal{B}(0, R_\varepsilon)} \dd|\mu| = \frac{\varepsilon}{2\Vert\mu\Vert_{\textrm{TV}}}|\mu|(\mathcal{B}(0, R_\varepsilon))\leq \frac{\varepsilon}{2}.
\end{align*}

And now we get that for any $\mu\in \mathcal{M}(\RRD)$ there exists an element $\mu^\varepsilon \in \mathrm{span}\{\delta_x\colon x\in\mathbb{Q}^d \}$ such that $\Vert \mu^{\varepsilon} - \mu \Vert_{(\CCA(\RRD))^*} \leq \varepsilon$.

Hence, $\textrm{span}\{\delta_x\colon  x\in \mathbb{Q}^d\}$ is a dense subset of $\mathcal{M}(\RRD)$. Countability of $\textrm{span}\{\delta_x\colon  x\in \mathbb{Q}^d\}$ is clear because of countability of $\mathbb{Q}^d$. 
This implies that the space $Z$ is separable.
\end{proof}

Moreover, we can characterize the dual space of $Z$, similar in spirit to \cite[Theorem 3.6, Theorem 3.7]{Hille-Worm:2009}. This result may be of separate interest in other settings.

Before giving and proving this characterization, we need the following lemma.

\begin{lemma}\label{lem:regu_deltabu}
The mapping defined by $\overline{\delta}(x)\colonequals \delta_x$ is  $\CCA(\RRD, Z)$.
\end{lemma}

\begin{proof}
For $f\in \CCA(\RRD)$, $\lambda\in\RRD$ and $x\in\RRD$ define $D\overline{\delta}(x)\in\mathcal{L}(\RRD,Z)$ by means of
\[
\pair{D\overline{\delta}(x)\lambda}{f}_{(\CCA(\RRD))^*, \CCA(\RRD)} \colonequals  \lambda\bullet \nabla f(x).
\]

By $\bullet$ we denote an inner product on $\RRD$. Thus, $\lambda\bullet\nabla f(x)$ relates {\cred to} the gradient of $f$ in the~direction given by $\lambda$. Then
\[
\frac{1}{|\lambda|}\bigl[ \delta_{x+\lambda}-\delta_x - D\overline{\delta}(x)\lambda\bigr] \to 0
\]
in $Z$ as $\lambda\to 0$. Thus $D\overline{\delta}(x)$ is the Fr\'echet derivative of $\overline{\delta}$ at $x$.

Of course, for $x,y\in\RRD$, $x\neq y$,
\begin{align*}
\|D\overline{\delta}(x) - D\overline{\delta}(y)\|_{Z} & =  \|D\overline{\delta}(x) - D\overline{\delta}(y)\|_{(\CCA(\RRD))^*}
\end{align*}
because $Z$ is linear subspace of $(\CCA(\RRD))^*$, thus $ \Vert \cdot\Vert_Z=\Vert \cdot \Vert_{(\CCA(\RRD))^*}$ coincides on $Z$. Now, we can estimate
\begin{align*}
\|D\overline{\delta}(x) - D\overline{\delta}(y)\|_{(\CCA(\RRD))^*}&=
\sup_{\lambda\in\RRD,\ |\lambda|\leq 1} \|D\overline{\delta}(x)\lambda - D\overline{\delta}(y)\lambda\|_{(\CCA(\RRD))^*}
\end{align*}
\begin{align*}
& = \sup_{\lambda\in\RRD,\ |\lambda|\leq 1} \sup_{\Vert f \Vert_{\CCA}\leq 1} \left|\pair{D\overline{\delta}(x)\lambda - D\overline{\delta}(y)\lambda}{f}_{(\CCA(\RRD))^*, \CCA(\RRD)}\right|\\
& = \sup_{\lambda\in\RRD,\ |\lambda|\leq 1} \sup_{\Vert f \Vert_{\CCA}\leq 1} \left|\sum_{i=1}^d \lambda_i\left(\partial_{x_i} f(x) - \partial_{x_i} f(y)\right) \right|\\
& \leq \sup_{\lambda\in\RRD,\ |\lambda|\leq 1} \sup_{\Vert f \Vert_{\CCA}\leq 1} |\lambda|\cdot\left(\sum_{i=1}^d \left|\partial_{x_i} f(x) - \partial_{x_i} f(y)\right|^2\right)^{1/2}\\
& \leq \sup_{\Vert f \Vert_{\CCA}\leq 1} \frac{|\nabla f(x) - \nabla f(y)|}{|x-y|^\alpha}\cdot |x-y|^\alpha\\
& \leq \sup_{\Vert f \Vert_{\CCA}\leq 1} \|f\|_{(\CCA(\RRD))^*}\cdot |x-y|^\alpha \ \leq\ |x-y|^\alpha.
\end{align*}

This concludes that $\|D\overline{\delta}(x) - D\overline{\delta}(y)\|_{Z} \leq |x-y|^\alpha$, thus 
 $\overline{\delta}\in \CCA(\RRD, Z)$.
\end{proof}

\begin{prop}\label{prop:Z_isom}
The space $Z^*$ is isomorphic to $\CCA(\RRD)$ under the map $\phi \mapsto T\phi$, where $T\phi(x)\colonequals \phi(\delta_x)$, $T\colon  Z^* \to \CCA(\RRD)$.
\end{prop}

\begin{proof}
We need to show that $T$ is bijection from $(Z^*, \Vert \cdot \Vert_{Z^*})$ to $(\CCA(\RRD), \Vert \cdot \Vert_{\CCA})$ such that
\[
T(\lambda_1 z_1^*+\lambda_2 z_2^*)=\lambda_1 T(z_1^*)+\lambda_2 T(z_2^*), 
\]
for $z_1^*, z_2^* \in Z^*$ and $\lambda_1, \lambda_2 \in \RRD$,  where
\[\Vert z^* \Vert_{Z^*}= \sup_{z\in Z} \left\lbrace |z^*( z)|\colon  \Vert z \Vert_Z \leq 1 \right\rbrace= \sup_{z \in Z}\{ z^*(z)\colon  \Vert z \Vert_Z\leq 1\}.\] In addition $T$ is bounded. By Banach Isomorphism Theorem, $T^{-1}$ is bounded.\\

{\bf Step 1.}
Obviously the mapping defined by $T\phi(x)=\phi(\delta_x)$ maps $Z^*$ into $\RR^{\RRD}$, where by $\RR^{\RRD}$ we denote a function space from $\RRD$ to $\RR$. 
The mapping $T$ is injective, because if $z_1^* \neq z_2^*$ then using density of $\textrm{span}\lbrace{ \delta_x\colon  x\in \RRD \rbrace}$ in $Z$ (Proposition \ref{prop:separable}) there exists $x\in \RRD$ such that
\[ z_1^*(\delta_x) \neq z_2^*(\delta_x) \Rightarrow \left(Tz_1^*\right)(x)\neq \left(Tz_2^*\right)(x)
.\]
Indeed,
\[
z_1^* \neq z_2^* \Rightarrow \exists z\in Z \textrm{ such that } z_1^*(z) \neq z_2^*(z).
\]

Since $\textrm{span}\lbrace{ \delta_x\colon  x\in \RRD \rbrace}$ is dense in Z, there exists $\lbrace{ z_n \rbrace}_{n\in \NN} \subset \textrm{span}\lbrace{ \delta_x\colon  x\in \RRD \rbrace}$ such that $z_n \to z$. Functionals $z_1^*, z_2^*$ are continuous and thus there exists $ n \textrm{ such that } z_1^*(z_n) \neq z_2^*(z_n)$. Of course $z_n= \sum_{i=1}^{k(n)} \alpha_i \delta_{x_i}$ and $z_1^*, z_2^*$ are linear
\[
\sum_{i=1}^{k(n)} \alpha_i z_1^*(\delta_{x_i}) \neq \sum_{i=1}^{k(n)} \alpha_i z_2^*(\delta_{x_i}).
\]

To show that the mapping $T$ is linear we need to show
\[T(\lambda_1 z_1^* + \lambda_2 z_2^*)= \lambda_1 Tz_1^* + \lambda_2 Tz_2^*, \qquad \textrm{ for all } \lambda_1, \lambda_2 \in \RRD, z_1^*, z_2^* \in Z^*,\]

what means that $\forall x\in \RRD$, $T(\lambda_1 z_1^* + \lambda_2 z_2^*)(x)= \lambda_1 Tz_1^*(x) + \lambda_2 Tz_2^*(x)$.
Indeed, $T(\lambda_1 z_1^* + \lambda_2 z_2^*)(x)= (\lambda_1 z_1^* + \lambda_2 z_2^*) (\delta_x)= \lambda_1 z_1^*  (\delta_x)+ \lambda_2 z_2^*(\delta_x)=\lambda_1 T(z_1^*)(x) + \lambda_2 T(z_2^*)(x).$

{\bf Step 2.} First we prove that $\textrm{im}(T(Z^*)) \subseteq \CCA(\RRD)$. By Lemma \ref{lem:regu_deltabu} we know that $(x \mapsto \delta_x) \in \CCA(\RRD, Z)$ and then $(x \mapsto z^*(\delta_x)) \in \CCA(\RRD,\RR)$ -- as a composition of two functions $(x \mapsto \delta_x) \in \CCA (\RRD, Z)$ and $z^* \in \mathcal{L}(Z, \RR)$. 
Therefore $\left(x \mapsto Tz^*(x)\right)\in \CCA(\RRD, \RR)$.

{\bf Step 3.} To prove the opposite inclusion $\CCA(\RRD)\subseteq \textrm{im}(T(Z^*))$, let us consider an arbitrary $y\in \CCA(\RRD)$. We want to show there exists $z^*_y$ such that $y=Tz^*_y$. Define a functional $z^*_y (\delta_x)\colonequals y(x)$. Our goal is to show that $z^*_y \in Z^*$.
It is enough to consider only $z\in \textrm{span} \lbrace{ \delta_x\colon  x\in \RRD \rbrace}$ and then
\[
|z^*_y(z)| = |z_y^*(\sum_{i=1}^n \alpha_i \delta_{x_i})|,
\]
functional $z^*_y$ is linear thus above is equal to $|\sum_{i=1}^n \alpha_i \cdot z^*_y(\delta_{x_i})|$. Using the definition of $z^*_y$ the~following holds
\begin{equation*}
\begin{split}
&\left|\sum_{i=1}^n \alpha_i  z^*_y(\delta_{x_i})\right|
=
\left|\sum_{i=1}^n \alpha_i y(x_i)\right|
=
\left|\sum_{i=1}^n \alpha_i \int_{\RRD} y \dd\delta_{x_i}\right|\\ 
&\qquad \qquad \qquad=
\left|\int_{\RRD} y \dd(\sum_{i=1}^n \alpha_i \delta_{x_i})\right|
= \left|\int_{\RRD} y \dd z\right| 
\leq 
\Vert y \Vert_{\CCA} \Vert z\Vert_{(\CCA(\RRD))^*}.
\end{split}
\end{equation*}

Thus $\Vert z^*_y \Vert_{Z^*}=\sup\lbrace{ z^*_y(z)\colon  \Vert z\Vert_Z \leq 1  \rbrace} \leq \Vert y\Vert_{\CCA(\RRD)}$.

{\bf Step 4.} To complete the proof we need continuity of the mapping $T$ which is of course equivalent to boundedness. In fact it is easy to see that $T^{-1}y=z^*_y $ is bounded. Estimations in step 3 imply that $\Vert T^{-1} \Vert \leq 1$. By Banach Isomorphism Theorem $\Vert T \Vert \leq C$, what finishes the~proof.
\end{proof}

\section{Acknowledgment}
This work was partially supported by the Polish Government MNiSW: P.G and A.\'{S}-G received support from the National Science Centre, UMO-2015/18/M/ST1/00075; K.\L ~acknowledges the support of from the National Science Centre, DEC-2012/05/E/ST1/02218.
\newpage

\bibliographystyle{alpha}   

\end{document}